\documentclass[10pt]{amsart}

\usepackage{amsmath,amsfonts,latexsym,amssymb,mathrsfs,setspace,enumerate,color,cite,graphicx,epsf,fullpage}

\newtheorem{thm}{Theorem}
\newtheorem{lem}[thm]{Lemma}

\newtheorem{cor}[thm]{Corollary}
\numberwithin{equation}{section}
\numberwithin{thm}{section}

\theoremstyle{definition}
\newtheorem{ex}{Example}

\newcommand{\rat}{\mathbb Q}
\newcommand{\real}{\mathbb R}

\newcommand{\alg}{\overline\rat}
\newcommand{\algt}{\alg^{\times}}

\newcommand{\intg}{\mathbb Z}
\newcommand{\nat}{\mathbb N}

\newcommand{\gal}{\mathrm{Gal}}
\newcommand{\norm}{\mathrm{Norm}}

\newcommand{\tors}{\mathrm{tors}}
\newcommand{\intL}{{\mathcal O}_L}
\newcommand{\intK}{{\mathcal O}_K}
\newcommand{\rank}{\mbox{rank}}

\title{Metric Mahler measures over number fields}
\author{Charles L. Samuels}
\address{Christopher Newport University, Department of Mathematics, 1 Avenue of the Arts, Newport News, VA 23606}
\email{charles.samuels@cnu.edu}

\begin{document}

\begin{abstract}
	For an algebraic number $\alpha$, the metric Mahler measure $m_1(\alpha)$ was first studied by Dubickas and Smyth in 2001 and was later generalized to the $t$-metric 
	Mahler measure $m_t(\alpha)$ by the author in 2010.  The definition of $m_t(\alpha)$ involves taking an infimum over a certain collection $N$-tuples of points in $\alg$, and 
	from previous work of Jankauskas and the author, the infimum in the definition of $m_t(\alpha)$ is attained by rational points when $\alpha\in \rat$.  As a consequence of 
	our main theorem in this article, we obtain an analog of this result when $\rat$ is replaced with any imaginary quadratic number field of class number equal to $1$.  
	Further, we study examples of other number fields to which our methods may be applied, and we establish various partial results in those cases.
\end{abstract}

\maketitle

\section{Introduction}

Suppose $L$ is a number field and $v$ is a place of $L$ dividing the place $p$ of $\rat$.  We shall write $L_v$ and $\rat_p$ to denote the completions of $L$ and $\rat$ with respect
to the $v$ and $p$.  Of course, we may view $\rat_p$ as a subfield of $L_v$ and we note the well-known identity
\begin{equation} \label{LocalGlobal}
	\sum_{v\mid p} [L_v:\rat_p] = [L:\rat].
\end{equation}
The right hand side of \eqref{LocalGlobal} is called the {\it global degree} while the summands on the left hand side are called {\it local degrees}.  We define the absolute value
$\|\ \|_v$ on $L_v$ to be the unique extension of the $p$-adic absolute value on $\rat_p$, and further, we define $|\ |_v$ by $|x|_v = \|x\|_v^{[L_v:\rat_p]/[L:\rat]}$ for all $x\in L_v$.
If $\alpha\in L$ and $p$ is a place of $\rat$, we note that
\begin{equation*}
	\prod_{v\mid p} |\alpha|_v = |\norm_{L/\rat}(\alpha)|_p^{1/[L:\rat]},
\end{equation*}
and consequently, we have the product formula
\begin{equation*}
	\prod_{v} |\alpha|_v = 1
\end{equation*}
for all $\alpha\in L^\times$.  Additionally, we define the {\it (logarithmic) Weil height} $h$ by
\begin{equation*}
	h(\alpha) = \sum_{v} \log\max\{1,|\alpha|_v\}
\end{equation*}
and observe that $h(\alpha)$ is independent of the choice of number field $L$ containing $\alpha$.  In this way, $h$ defines a map from $\alg$ to $[0,\infty)$ which satisfies
\begin{enumerate}
	\item $h(\alpha^n) = |n|\cdot h(\alpha)$ for all $n\in \intg$ and all $\alpha\in \algt$
	\item $h(\alpha\beta) \leq h(\alpha) + h(\beta)$ for all $\alpha,\beta\in \algt$.
	\item If $\alpha$ and $\beta$ are Galois conjugates over $\rat$ then $h(\alpha) = h(\beta)$.
	\item $h(\zeta\alpha) = h(\alpha)$ for all $\alpha\in \algt$ and all roots of unity $\zeta$.
\end{enumerate}
If $K$ is another number field and $\alpha\in\algt$ then the {\it Mahler measure of $\alpha$ over $K$} is defined by
\begin{equation*} \label{measure}
	m_K(\alpha) = [K(\alpha):K]\cdot h(\alpha).
\end{equation*}
The Mahler measure over $\rat$ has a long history dating back to a 1933 problem of D.H. Lehmer \cite{Lehmer} which asks whether there exists a constant $c > 0$ such that
$m_\rat(\alpha) \geq c$ for all non-torsion points $\alpha\in \algt$.  A variety of authors have established partial results in the direction of Lehmer's problem
(see \cite{BDM,Dobrowolski,Schinzel,Smyth,Voutier}, for instance) although the general case remains open.

If $t$ is a positive real number, we define the {\it $t$-metric Mahler measure of $\alpha$ over $K$} by
\begin{equation*} \label{TMetric}
	m_{K,t}(\alpha) = \inf\left\{ \left(\sum_{n=1}^N m_K(\alpha_n)^t\right)^{1/t}:N\in \nat,\ \alpha_n\in \alg,\ \alpha = \prod_{n=1}^N\alpha_n\right\}.
\end{equation*}
It is straightforward to verify that $m_{K,t}(\alpha\beta)^t \leq m_{K,t}(\alpha)^t + m_{K,t}(\beta)^t$ for all $\alpha,\beta\in \algt$ and all $t > 0$.  As a result, the map 
$(\alpha,\beta)\mapsto m_{K,t}(\alpha\beta^{-1})^t$ creates a well-defined metric on $\algt/\algt_\tors$.  Additionally, if $t>0$ and $\phi:\algt\to [0,\infty)$ is any function satisfying 
$0\leq \phi \leq m_K$ and $\phi(\alpha\beta)^t \leq \phi(\alpha)^t + \phi(\beta)^t$ for all $\alpha,\beta\in \algt$, then $\phi \leq m_{K,t}$.  For these reasons, we often think of $m_{K,t}$ as a maximal
metric version of the Mahler measure.   In the expected way, we also define the {\it $\infty$-metric Mahler measure over $K$} by
\begin{equation*}
	m_{K,\infty}(\alpha) =  \inf\left\{\max_{1\leq n\leq N} \{m_K(\alpha_n)\}:N\in \nat,\ \alpha_n\in \alg,\ \alpha = \prod_{n=1}^N\alpha_n\right\}.
\end{equation*}
It is clear from the definition that $\lim_{t\to\infty} m_{K,t}(\alpha) = m_{K,\infty}(\alpha)$.

The $t$-metric Mahler measures over $\rat$ have been studied extensively by Dubickas, Smyth, Fili, Jankauskas and the author in an assortment of previous articles
\cite{DubSmyth2,DubSmyth,FiliSamuels,JankSamuels,SamuelsInfimum,SamuelsCollection,SamuelsParametrized,SamuelsMetric,SamuelsFibonacci,SamuelsFibonacci2}.
For example, the author \cite{SamuelsCollection} showed that the infimum in the definition of $m_{\rat,t}(\alpha)$ is attained for all $\alpha\in \alg$ and all $t >0$.  
Subsequent articles established various techniques for finding points $(\alpha_1,\alpha_2,\cdots,\alpha_N)\in \alg$ which attain the infimum in $m_{\rat,t}(\alpha)$ for different values of $t$.

In the present paper, we are particularly interested in
an article of Jankauskas and the author \cite{JankSamuels} which establishes that the infimum in $m_{\rat,t}(\alpha)$ is attained by rational points when $\alpha\in \rat$.  Our goal is to study 
the extent to which this result generalizes to the metric Mahler measure over a number field.

The proof technique of \cite{JankSamuels} utilizes roughly the following outline.  If $\alpha_1,\alpha_2,\ldots,\alpha_N\in \alg$ are such that $\alpha = \alpha_1\alpha_2\cdots\alpha_N$,
we identify $\gamma_0,\gamma_1,\ldots,\gamma_N\in \rat$ such that
\begin{enumerate}[(i)]
	\item $\alpha= \gamma_0\gamma_1\cdots\gamma_N$
	\item $\gamma_0$ is a root of unity
	\item $m_\rat(\gamma_n) \leq m_\rat(\alpha_n)$ for all $1\leq n\leq N$.
\end{enumerate}
It is a simple consequence of these facts that the infimum in $m_{\rat,t}(\alpha)$ is attained by points in $\rat$.  Unfortunately, the method for constructing the points $\gamma_n$
uses various elementary divisibility properties of $\rat$ that are not present in a general number field.  Therefore, those methods will need substantial modification in 
order to yield analogous results for $m_{K,t}(\alpha)$.  We shall require a new definition.

We say that a number field $K$ is {\it balanced} if for every non-zero point $x\in \intK$ there exists a unit $u\in \intK$ such that $|ux|_v \geq 1$ for all Archimedean places $v$ of $K$.
If there exists $x\in \intK$ for which no such unit exists, then $K$ is called {\it unbalanced}.  Our main result is a generalization of the proof technique in \cite{JankSamuels} described above.

\begin{thm} \label{GeneralReplacement}
	Suppose that $K$ is a number field whose Hilbert class field $L$ is balanced.  Assume that $\alpha\in K^\times$ and $\alpha_1,\alpha_2,\ldots,\alpha_N\in\alg$ are such that 
	$\alpha = \alpha_1\alpha_2\cdots\alpha_N$.  Then there exist $\gamma_0,\gamma_1,\ldots,\gamma_N\in L$ satisfying the following three conditions:
	\begin{enumerate}[(i)]
		\item\label{Product} $\alpha= \gamma_0\gamma_1\cdots\gamma_N$
		\item\label{Unit} $\gamma_0$ is a unit in $\intL$
		\item\label{Measures} $h(\gamma_n) \leq m_K(\alpha_n)$ for all $1\leq n\leq N$.
	\end{enumerate}
\end{thm}

We remind the reader that the Hilbert class field $L$ of $K$ is the maximal Abelian unramified extension of $K$.  It is well-known that $\gal(L/K)$ is isomorphic to the ideal class group of $K$ (this is
a special case of \cite[Theorem 0.3]{Milne}), so in particular, $[L:K]$ is equal to the class number of $K$.  An important special case of Theorem \ref{GeneralReplacement} arises when $K$ has class number 
equal to $1$, in which case $L=K$ and $h(\gamma_n) = m_K(\gamma_n)$ for all $1\leq n\leq N$.  These observations give rise to a useful corollary.

\begin{cor} \label{ClassNumber1}
	Let $K$ be a balanced number field of class number equal to $1$.  Assume that $\alpha\in K^\times$ and $\alpha_1,\alpha_2,\ldots,\alpha_N\in\alg$ are such that 
	$\alpha = \alpha_1\alpha_2\cdots\alpha_N$.  Then there exist $\gamma_0,\gamma_1,\ldots,\gamma_N\in K$ satisfying the following three conditions:
	\begin{enumerate}[(i)]
		\item $\alpha= \gamma_0\gamma_1\cdots\gamma_N$
		\item$\gamma_0$ is a unit in $\intK$
		\item $m_K(\gamma_n) \leq m_K(\alpha_n)$ for all $1\leq n\leq N$.
	\end{enumerate}
\end{cor}

The simplest examples of balanced number fields are the rational numbers and the imaginary quadratic extensions of $\rat$.  All such fields have exactly one Archimedean place $w$, so the product
formula forces $|x|_w \geq 1$ for all $x\in \intK$.  As a result, we may simply use $u=1$ to satisfy the definition of balanced.  Additionally, in all such fields, a point is a unit if and only if it is a root 
of unity.  Hence, $m_{K}(\gamma_0) = 0$ and we obtain the following direct generalization of \cite[Theorem 1.2]{JankSamuels}.

\begin{cor}\label{KInfimum}
	Suppose that $K=\rat$ or $K$ is an imaginary quadratic extension of $\rat$ of class number $1$.  If $\alpha\in K$ and $t > 0$ then there exist points $\alpha_1,\alpha_2,\ldots,\alpha_N\in K$
	such that $\alpha = \alpha_1\alpha_2\cdots\alpha_N$ and 
	\begin{equation*}
		m_{K,t}(\alpha)^t = \sum_{n=1}^N m_K(\alpha_n)^t.
	\end{equation*}
	Similarly, there exist $\alpha_1,\alpha_2,\ldots,\alpha_N\in K$ such that $\alpha = \alpha_1\alpha_2\cdots\alpha_N$ and $m_{K,\infty}(\alpha) =\max\{m_K(\alpha_n):1\leq n\leq N\}$.
\end{cor}

According to \cite[\S 1.6]{Neukirch}, the imaginary quadratic extensions of $\rat$ with class number $1$ are known to be $\rat(\sqrt{-d})$, where $d\in \{1,2,3,7,11,19,43,67,163\}$.  
These nine fields, along with the rational numbers, form the complete list of number fields that are covered by Corollary \ref{KInfimum}.  In our next section, we shall explore additional
examples where Corollary \ref{ClassNumber1} may be applied.

\section{Additional Examples of Balanced Number Fields}\label{Examples}

We let $\intK^\times$ denote the group of units in $\intK$ and remind the reader that there exists a non-negative integer $r$ such that $\intK^\times \simeq \intg^r\oplus K^\times_\tors$.
In this notation, $r$ is called the {\it rank of $\intK^\times$} (or simply the {\it rank of $K$}) and is denoted $r = \rank(K)$.  It follows from Dirchlet's 
Unit Theorem (see \cite[\S 1.7]{Neukirch}, for instance) that $\rank(K)$ is one less than the number of Archimedean places of $K$.  
For example, we have $r=0$ if and only if either $K=\rat$ or $K$ is an imaginary quadratic extension of $\rat$.  As we have noted prior to Corollary \ref{KInfimum}, all of these fields
are balanced. 

The situation becomes slightly more complicated when $\rank(K)=1$.  This scenario occurs in precisely the following three cases:
\begin{enumerate}[(a)]
	\item $K$ is real quadratic extension of $\rat$
	\item $K$ is a cubic extension of $\rat$ which is not totally real
	\item $K$ is a totally imaginary quartic extension of $\rat$.
\end{enumerate}
The following lemma is useful for producing balanced rank $1$ number fields.

\begin{lem} \label{Rank1}
	Suppose that $K$ is a number field of rank $1$ and $\xi\in\intK^\times\setminus K^\times_\tors$.  If $x\in\intK$ is such that $[K:\rat]\cdot h(\xi) \leq \log |\norm_{K/\rat}(x)|_\infty$
	then there exists a unit $u\in \intK$ such that $|ux|_v \geq 1$ for all Archimedean places $v$ of $K$.  In particular, if
	\begin{equation} \label{PrelimBalanced}
		1 < [K:\rat]\cdot h(\xi) \leq \log\min\{|\norm_{K/\rat}(y)|_\infty: y\in \intK\setminus\intK^\times\}.
	\end{equation}
	then $K$ is balanced.
\end{lem}

Unfortunately, our results are not enough to obtain a result as strong as Corollary \ref{KInfimum} for number fields of rank $1$.
Indeed, the unit $\gamma_0$ which arises from Corollary \ref{ClassNumber1} may not be a root of unity, and hence, it could have non-zero Mahler measure.
However, we are able to obtain a partial result dealing with the case $t=\infty$.

\begin{cor} \label{KInfimumWeak}
	Suppose that $K$ is a number field of rank $1$ and class number $1$.  Further assume that there exists a unit $\xi$ of $K$ such that 
	\begin{equation} \label{WeakHeight}
			1 < [K:\rat]\cdot h(\xi) \leq \log\min\{|\norm_{K/\rat}(y)|_\infty: y\in \intK\setminus\intK^\times\}.
	\end{equation}
	If $\alpha\in K^\times\setminus \intK^\times$ then there exist $\alpha_1,\alpha_2,\ldots,\alpha_N\in K$ 
	such that $\alpha = \alpha_1\alpha_2\cdots\alpha_N$ and 
	\begin{equation*}
		m_{K,\infty}(\alpha) = \max\left\{m_K(\alpha_n): 1\leq n\leq N\right\}.
	\end{equation*}
\end{cor}

Luckily, there is a standard recipe for creating number fields $K$ satisfying the hypotheses of Corollary \ref{KInfimumWeak}.  Select an irreducible polynomial 
$f \in \intg[x]$ having constant term equal to $\pm 1$ satisfying one of the following three properties:
\begin{enumerate}[(i)]
	\item $\deg f = 2$ and $f$ has a real root $\xi$ with $1 < |\xi|_\infty < 2$
	\item $\deg f = 3$ and $f$ has a unique real root $\xi$ such that $1 < |\xi|_\infty < 2$
	\item $\deg f = 4$ and $f$ has four imaginary roots with one of those roots $\xi$ satisfying $1 < |\xi|_\infty < \sqrt 2$.
\end{enumerate}
In these cases, $\rat(\xi)$ must satisfy the hypotheses of Corollary \ref{KInfimumWeak}.  For instance, we could use $f(x) = x^2 - x - 1$.  Then the golden ratio $\xi = (1+\sqrt 5)/2$
is a root of $f$, and thus, $K = \rat(\xi)$ satisfies \eqref{WeakHeight}.  Similarly, we may set $f(x) = x^3 - x -1$ so that $f$ is the famous polynomial studied by Chris Smyth in \cite{Smyth}.  
In this case, $f$ has exactly one real root $\xi =1.32\ldots$, and therefore, $K = \rat(\xi)$ also satisfies \eqref{WeakHeight}.  In both of these cases, these number fields are known to have class 
number equal to $1$ so that Corollary \ref{KInfimumWeak} applies.

We conclude this section by providing the reader with two additional examples of rank $1$ number fields.  First, we give an example of an unbalanced number field, and second, we give an example showing that the converse
of the second statement of Lemma \ref{Rank1} is false.

\begin{ex}
	We claim that $K = \rat(\sqrt 3)$ is not balanced.  To see this, we must identify a non-zero point $x\in \intK$ for which there is no unit $u\in \intK$ satisfying $|ux|_v \geq1$ for all $v\mid \infty$.
	Since $K$ is a real quadratic number field, it must have exactly two Archimedean places $w_1$ and $w_2$.  Moreover, since $3\not\equiv 1\mod 4$, 
	it is well-known that $\intK = \intg[\sqrt 3]$ (see \cite[\S 2.7]{Jarvis}).  As a result, we may assume without loss of generality that
	\begin{equation*}
		\|a+b\sqrt 3\|_{v_1} = |a+b\sqrt 3|\quad\mbox{and}\quad \|a+b\sqrt 3\|_{v_2} = |a-b\sqrt 3|,
	\end{equation*}
	where $|\ |$ denotes the usual absolute value on $\real$ and $\sqrt 3$ is the positive square root of $3$.  Additionally, $\rank(K) = 1$ so that $\intK^\times/K^\times_\tors$ is
	cyclic.  Using the technique described in \cite[\S 6.4 and \S 6.5]{Cohn}, we find that $\xi = 2 + \sqrt 3$ is a generator of this group.\footnote{$\xi$ is commonly called a {\it fundamental unit}.}
	
	Now let $x = 1 + \sqrt 3$ and assume that $u$ is a unit in $K$ such that $\|ux\|_{v_1} \geq 1$ and $\|ux\|_{v_2} \geq 1$.  There exists $\ell\in \intg$ 
	such that $u = \pm \xi^\ell$.  Thus
	\begin{equation*}
		1\leq \|\xi\|^\ell_{v_1}\|x\|_{v_1} = \|2 + \sqrt 3\|^\ell_{v_1}\|1 + \sqrt 3\|_{v_1} = (2+\sqrt 3)^\ell(1 +\sqrt 3) < (2+\sqrt 3)^{\ell + 1},
	\end{equation*} 
	which forces $\ell > 0$ and implies that
	\begin{equation*}
		1 \leq \|\xi\|^\ell_{v_2}\|x\|_{v_2} = (2- \sqrt 3)^\ell (\sqrt 3 - 1) < 1^\ell\cdot 1 = 1,
	\end{equation*}
	a contradiction. 
\end{ex}

\begin{ex}
	We now assert that $K = \rat(\sqrt 2)$ is balanced even though there is no unit $\xi \in K$ satisfying \eqref{PrelimBalanced}.  First, we note that
	$1 + \sqrt 2$ is a fundamental unit of $K$ and $h(1+\sqrt 2) = \frac{1}{2} \log (1 + \sqrt 2)$.  This implies that
	\begin{equation*}
		\log|\norm_{K/\rat}(2 + \sqrt 2)|_\infty = \log 2 < \log (1+\sqrt 2)  = [K:\rat]\cdot h(1+\sqrt 2).
	\end{equation*}
	If $\xi$ is another unit but not a root of unity, then there must exist a non-zero integer $\ell$ such that $\xi = \pm (1+\sqrt 2)^\ell$.   It follows that
	\begin{equation*}
		[K:\rat]\cdot h(\xi) \geq [K:\rat]\cdot h(1+\sqrt 2) > \log|\norm_{K/\rat}(2 + \sqrt 2)|_\infty,
	\end{equation*}
	so that $K$ fails to satisfy \eqref{PrelimBalanced} for any unit $\xi\in \intK$.  
	
	To see that $K$ is balanced, we assume that $x$ is a non-zero point in $\intK$.  If $x$ is unit then we use $u=x^{-1}$ to satisfy the definition of balanced.  If $|\norm_{K/\rat}(x)|_\infty \geq 3$ then 
	we have $[K:\rat]\cdot h(1+\sqrt 2) \leq \log |\norm_{K/\rat}(x)|_\infty$ and we may apply the first statement of Lemma \ref{Rank1}.  Therefore, it remains only to consider the case that
	$|\norm_{K/\rat}(x)|_\infty = 2$. 
	
	Since $\intK = \intg[\sqrt 2]$ we may write $x = a + b\sqrt 2$, where $a,b\in \intg$, and since we have assumed that $|\norm_{K/\rat}(x)|_\infty = 2$, we get
	$a^2 -2b^2 = \pm 2$.  It follows now that $a$ is even and 
	\begin{equation*}
		\left(b + \frac{a}{2}\sqrt 2\right)\left(b - \frac{a}{2}\sqrt 2\right) = b^2 - 2\left(\frac{a}{2}\right)^2 = \pm 1
	\end{equation*}
	which implies that $b + \frac{a}{2}\sqrt 2$ is a unit in $\intK$.  Now setting $u = (b + \frac{a}{2}\sqrt 2)^{-1}$ we get that
	\begin{equation*}
		x = a + b\sqrt 2 = u^{-1}\sqrt 2
	\end{equation*}
	and it follows that $||ux||_v  = \sqrt 2 > 1$ for all $v\mid\infty$.
		
\end{ex}

\section{Proofs of Main Results}\label{KRat}

The proof of Theorem \ref{GeneralReplacement} makes use of fractional ideals so we take a few moments to remind the reader of the relevant facts and notation (see \cite[p. 760]{DF} for further detail
than what is provided here).  Suppose that $R$ is an integral domain with field of fractions $K$.  An $R$-submodule $I$ of $K$ is called a {\it fractional ideal of $R$} if there exists 
$d\in R\setminus\{0\}$ such that $dI \subseteq R$.  Of course, every ideal of $R$ is a fractional ideal and such ideals are sometimes called {\it integral ideals}.    If there exists a fractional
ideal $J$ of $R$ such that $IJ = R$ then we say that $I$ is {\it invertible} and that $J$ is the {\it inverse of $I$}, denoted $J = I^{-1}$.

If $R$ is a subring of another integral domain $S$ and $I$ is a fractional ideal of $R$, we define the {\it extension of $I$ to $S$} by
\begin{equation*}
	IS = \left\{ \sum_{n=1}^N a_ns_n: N\in \nat,\ a_n\in I,\ s_n\in S\right\}.
\end{equation*}
It is easily verified that $IS$ equals the intersection of all fractional ideals of $S$ which contain $I$.  Moreover, we note a series of straightforward facts regarding extensions of fractional ideals.

\begin{lem} \label{FractionalIdeals}
	Suppose that $R$ and $S$ are integral domains such that $R$ is a subring of $S$ and assume that $K$ is the field of fractions of $R$.
	\begin{enumerate}[(i)]
		\item\label{ProdExtend} If $I$ and $J$ are fractional ideals of $R$ then $(IJ)S = (IS)(JS)$.
		\item\label{ExtendInverse} If $I$ is an invertible fractional ideal of $R$ then $IS$ is invertible and $(IS)^{-1} = I^{-1}S$.
		\item\label{PrincipalExtend} If $\alpha\in K$ then $(\alpha R)S = \alpha S$
		\item\label{CoPrimeExtend} If $I$ and $J$ are integral ideals of $R$ such that $I+J = R$ then $IS + JS = S$.
		\item\label{RatioContain} If $I,I'J,J'$ are integral ideals of $R$ with $I+I' = R$ and $IJ' = I'J$ then $J\subseteq I$ and $J'\subseteq I'$.
	\end{enumerate}
\end{lem}
\begin{proof}
	Assuming that $x\in (IJ)S$ we select $x_n\in IJ$, $s_n\in S$ and $N\in \nat$ such that 
	\begin{equation*}
		x = \sum_{n=1}^N x_ns_n.
	\end{equation*}
	Additionally, we let $a_{n,i}\in I$, $b_{n,i}\in J$ and $k_n\in \nat$ be such that
	\begin{equation*}
		x_n = \sum_{i=1}^{k_n} a_{n,i}b_{n,i}
	\end{equation*}
	which yields that
	\begin{equation*}
		x = \sum_{n=1}^N \left( \sum_{i=1}^{k_n} a_{n,i}b_{n,i}\right)s_n = \sum_{n=1}^N\sum_{i=1}^{k_n} (a_{n,i}\cdot 1)\cdot (b_{n,i}s_n) \in (IS)(JS).
	\end{equation*}
	Now let $x\in (IS)(JS)$ so that there exist $x_n\in IS$, $y_n\in JS$ and $N\in \nat$ such that
	\begin{equation*}
		x = \sum_{n=1}^Nx_ny_n
	\end{equation*}
	Next we let $k_n,\ell_n\in \nat$, $a_{n,i}\in I$, $b_{n,j}\in J$, $r_{n,i},s_{n,j}\in S$ such that 
	\begin{equation*}
		x_n = \sum_{i=1}^{k_n} a_{n,i}r_{n,i}\quad\mbox{and}\quad y_n = \sum_{j=1}^{\ell_n} b_{n,j}s_{n,j}.
	\end{equation*}
	This means that
	\begin{equation*}
		x = \sum_{n=1}^N \left(\sum_{i=1}^{k_n} a_{n,i}r_{n,i}\right)\left(\sum_{j=1}^{\ell_n} b_{n,j}s_{n,j}\right) = \sum_{n=1}^N \sum_{i=1}^{k_n}\sum_{j=1}^{\ell_n} a_{n,i}b_{n,j}r_{n,i}s_{n,j} \in (IJ)S
	\end{equation*}
	establishing \eqref{ProdExtend}, and \eqref{ExtendInverse} follows by applying \eqref{ProdExtend} with $J = I^{-1}$.
	
	For \eqref{PrincipalExtend}, we clearly have $\alpha S\subseteq (\alpha R)S$.  If $x\in (\alpha R)S$ we write
	\begin{equation*}
		x = \sum_{n=1}^N \alpha r_n s_n
	\end{equation*}
	for some $r_n\in R$ and $s_n\in S$.  Thus
	\begin{equation*}
		x = \alpha\sum_{n=1}^N r_ns_n \in \alpha S
	\end{equation*}
	verifying \eqref{PrincipalExtend}.
	
	For \eqref{CoPrimeExtend} we write $1 = a + b$ for some $a\in I$ and $b\in J$ so that $1 = a\cdot 1 + b\cdot 1\in IS + JS$.  But $IS+JS$ is an integral ideal of $S$ so that 
	$IS + JS = S$.  To verify \eqref{RatioContain}, we observe that $IJ' + IJ = I'J + IJ$.  Using the distributive law for ideal multiplication (see \cite[\S 7.3, Exercise 35(a)]{DF}), we get that
	\begin{equation*}
		I(J' + J) = (I'+I)J = RJ = J,
	\end{equation*}
	and we conclude that $J \subseteq I$.  A similar argument establishes that $J'\subseteq I'$ completing the proof.
\end{proof}

If $R$ is a Dedekind domain then every fractional ideal of $R$ is invertible, and if $I$ and $J$ are fractional ideals of $R$, we shall write $I/J = IJ^{-1}$.
We caution the reader that, in our notation, $R/I$ is simply an alternate way of writing $I^{-1}$ and does not refer to a quotient ring.
Still assuming that $R$ is Dedekind domain, every integral ideal may be factored uniquely into prime ideals of $R$.
A pair of integral ideals $I$ and $J$ have no common prime factors if and only if $I+J = R$, and in this case, $I$ and $J$ are called {\it relatively prime}.  

If $A$ is a fractional ideal of any domain $R$, then $A$ must have the form $d^{-1} I$ for some integral ideal $I$ of $R$ and $d\in R\setminus\{0\}$.
As a result, we see that $A = (d^{-1}R)I= (dR)^{-1} I = I/(dR)$, so in particular, $A$ is a ratio of integral ideals.  Further assuming that $R$ is a Dedekind domain, there must exist a 
relatively prime pair of integral ideals $I$ and $J$ such that $A = I/J$.

If $I,J_1,J_2,\ldots,J_N$ are integral ideals of a Dedekind domain $R$ such that $J_1J_2\cdots J_N \subseteq I$ then our proof of Theorem \ref{GeneralReplacement} requires that 
we identify a set of integral ideals $I_1,I_2,\ldots,I_N$ such that $I = I_1I_2\cdots I_N$ and $J_n \subseteq I_n$ for all $1\leq n\leq N$.  The following lemma, which is analogous to
\cite[Lemma 2.2]{JankSamuels}, shows a method for constructing the ideals $I_n$.

\begin{lem} \label{DedekindDomainAlgorithm}
	Suppose that $R$ is a Dedekind domain and $I,I_1,\ldots,I_N,J_1,\ldots, J_N$ are integral ideals of $R$ such that
	\begin{enumerate}[(i)]
		\item\label{FirstDivisibility} $J_1J_2\cdots J_N \subseteq I$
		\item\label{ProgressDivisibility} $I_1I_2\cdots I_n = I_1I_2\cdots I_{n-1} J_n + I$ for all $1\leq n\leq N$.
	\end{enumerate}
	Then $I = I_1I_2\cdots I_N$ and $J_n \subseteq I_n$ for all $1\leq n\leq N$.
\end{lem}
\begin{proof}
	To establish the first conclusion of the lemma, we shall first prove by induction on $n$ that 
	\begin{equation} \label{PrelimIdealEquality}
		I_1I_2\cdots I_n = J_1J_2\cdots J_n + I
	\end{equation}
	for all $1\leq n\leq N$.  If $n=1$ then \eqref{ProgressDivisibility} implies that $I_1 = J_1 + I$ so the base case is obtained immediately.
	Now assuming that $2\leq n\leq N$ and $I_1I_2\cdots I_{n-1} = J_1J_2\cdots J_{n-1} + I$,
	we may multiply both sides of this equality by $J_n$ and use the distributive law for ideal multiplication to conclude that 
	\begin{equation*}
		I_1I_2\cdots I_{n-1}J_n = J_1J_2\cdots J_{n-1}J_n + IJ_n.
	\end{equation*}
	Now substitute into \eqref{ProgressDivisibility} to deduce that 
	\begin{equation*}
		I_1I_2\cdots I_n = J_1J_2\cdots J_n + IJ_n +I = J_1J_2\cdots J_n + I(J_n + R) = J_1J_2\cdots J_n + I
	\end{equation*}
	establishing \eqref{PrelimIdealEquality}.  Now by applying \eqref{PrelimIdealEquality} with $n=N$ we get that
	\begin{equation*}
		I_1I_2\cdots I_N = J_1J_2\cdots J_N + I
	\end{equation*}
	and the first conclusion of the lemma follows from \eqref{FirstDivisibility}.
	
	To verify the second conclusion, we observe that
	\begin{equation*}
		I_1I_2\cdots I_{n-1} J_n \subseteq I_1I_2\cdots I_{n-1} J_n + I  = I_1I_2\cdots I_{n-1} I_n.
	\end{equation*}
	Multiplication by fractional ideals preserves set containment, so the result follows by multiplying both sides by the inverse of $I_1 I_2\cdots I_{n-1}$.
\end{proof}

Under the assumption $J_1J_2\cdots J_N \subseteq I$, Lemma \ref{DedekindDomainAlgorithm} provides an algorithm for creating a set 
of ideals $I_1,I_2,\ldots, I_N$ satisfying $I= I_1I_2\cdots I_N$ and $J_n \subseteq I_n$ for all $1\leq n\leq N$.  Indeed, all ideals of a Dedekind domain are invertible fractional ideals
and it is easily shown by induction on $n$ that $I_1^{-1} I_2^{-1}\cdots I_{n-1}^{-1}I$ is an integral ideal of $R$.  Hence, we may define
\begin{equation*}
	I_n = J_n + I_1^{-1} I_2^{-1}\cdots I_{n-1}^{-1} I
\end{equation*}
so that the ideals $I_1,I_2,\ldots,I_n$ satisfy \eqref{ProgressDivisibility} in Lemma \ref{DedekindDomainAlgorithm}.

The reader has perhaps noticed that the first conclusion of Lemma \ref{DedekindDomainAlgorithm} does not require that $R$ be a Dedekind domain.  Indeed,
one can obtain that $I = I_1I_2\cdots I_N$ by assuming only that $R$ is a commutative ring with unity.  On the other hand, we are required to assume that $I_1,I_2,\ldots,I_{n-1}$ be invertible
as fractional ideals of $R$ in order to deduce that $J_n\subseteq I_n$.
For instance, let $\sqrt[k]{2}$ denote the positive real $k$th root of $2$ and let $R = \intg[\sqrt 2, \sqrt[3]{2},\sqrt[4]{2},\ldots]$.  Now define $I = (\sqrt 2, \sqrt[3]{2},\sqrt[4]{2},\ldots)$
and let $J_1 = I_1 = I_2 = I$ and  $J_2 = R$.  Directly from these definitions, we obtain that  $J_1J_2 \subseteq I$ and $I_1 = J_1 + I$.
If $a\in I$ then write $a = \sum_{k=1}^M a_k \sqrt[k]{2}$ and observe that
\begin{equation*}
	a = \sum_{k=1}^M a_k \sqrt[2k]{2}\sqrt[2k]{2} \in I^2.
\end{equation*}
It now follows that $I^2 = I$ and we obtain
\begin{equation*}
	I_1I_2  = I^2 = I = I+I = I_1R + I = I_1J_2 + I.
\end{equation*}
As a result, the ideals $I,I_1,I_2,J_1$ and $J_2$ satisfy the assumptions of Lemma \ref{DedekindDomainAlgorithm} but do not satisfy the second conclusion that $J_2\not\subseteq I_2$.
Therefore, we do indeed require the assumption that $R$ is a Dedekind domain in order to obtain the full statement of Lemma \ref{DedekindDomainAlgorithm}.

\begin{proof}[Proof of Theorem \ref{GeneralReplacement}]
	Assume that $E$ is a Galois extension of $K$ containing $\alpha_1,\alpha_2,\ldots,\alpha_N$.
	Since we know that $\alpha = \alpha_1\alpha_2\cdots\alpha_N$ we get immediately
	\begin{equation*}
		\alpha^{[E:K]} = \prod_{n=1}^N \norm_{E/K}(\alpha_n) = \prod_{n=1}^N \norm_{K(\alpha_n)/K}(\alpha_n)^{[E:K(\alpha_n)]}
	\end{equation*}
	For simplicity, we shall set $\beta_n = \norm_{K(\alpha_n)/K}(\alpha_n)$ so that $\beta_n\in K$ and
	\begin{equation} \label{AlphaPoints}
		\alpha^{[E:K]}  =  \prod_{n=1}^N \beta_n^{[E:K(\alpha_n)]}
	\end{equation}
	We now define fractional ideals $A, B_1,B_2,\cdots, B_N$ of $R$ by $A = \alpha \intK$ and $B_n = \beta_n \intK$ for all $n$.  It is easily verified that $(a\intK)(b\intK) = (ab)\intK$ for 
	all $a,b\in K$, and therefore, we obtain that 
	\begin{equation*}
		A^{[E:K]} = \alpha^{[E:K]} \intK \quad\mbox{and}\quad \prod_{n=1}^N B_n^{[E:K(\alpha_n)]} = \left( \prod_{n=1}^N \beta_n^{[E:K(\alpha_n)]}\right) \intK.
	\end{equation*}
	Now using \eqref{AlphaPoints} we conclude that
	\begin{equation} \label{AlphaIdeals}
		A^{[E:K]} = \prod_{n=1}^N B_n^{[E:K(\alpha_n)]}.
	\end{equation}
	Now let $I$ and $I'$ be relatively prime integral ideals of $\intK$ such that $I/I' = A$.  Similarly, define $J_n$ and $J'_n$
	to be relatively prime ideals of $\intK$ such that $J_n/J'_n = B_n$.  In view of these definitions, \eqref{AlphaIdeals} yields
	\begin{align*}
		\frac{I^{[E:K]}}{(I')^{[E:K]}} & = \left( \frac{I}{I'}\right) ^{[E:K]}   = \prod_{n=1}^N\left( \frac{J_n}{J'_n} \right)^{[E:K(\alpha_n)]}
			 = \left(\prod_{n=1}^N J_n^{[E:K(\alpha_n)]}\right) \bigg/ \left(\prod_{n=1}^N (J'_n)^{[E:K(\alpha_n)]}\right)
	\end{align*}
	and Lemma \ref{FractionalIdeals}\eqref{RatioContain} gives
	\begin{equation*}
		\prod_{n=1}^N J_n^{[E:K(\alpha_n)]}  \subseteq I^{[E:K]}\quad\mbox{and}\quad \prod_{n=1}^N (J'_n)^{[E:K(\alpha_n)]}  \subseteq (I')^{[L:K]}.
	\end{equation*}
	Then it follows that
	\begin{equation*}
		\left(\prod_{n=1}^N J_n\right)^{[E:K]}  \subseteq I^{[E:K]}\quad\mbox{and}\quad \left(\prod_{n=1}^N J'_n\right)^{[E:K]}  \subseteq (I')^{[E:K]},
	\end{equation*}
	and using the fact the $\intK$ has unique factorization of ideals into prime ideals, we deduce that
	\begin{equation*}
		\prod_{n=1}^N J_n \subseteq I \quad\mbox{and}\quad \prod_{n=1}^N J'_n \subseteq I'.
	\end{equation*}
	Now applying Lemma \ref{DedekindDomainAlgorithm}, we obtain integral ideals $I_1,I_2,\ldots,I_N,I'_1,I'_2,\ldots,I'_N$ of $\intK$ satisfying the following properties:
	\begin{enumerate}[(a)]
		\item\label{IProducts} $I=I_1I_2\cdots I_N$ and $I' = I'_1I'_2\cdots I'_N$
		\item\label{IContainments} $J_n\subseteq I_n$ and $J'_n\subseteq I'_n$ for all $n$.
	\end{enumerate}
	This enables us to conclude from Lemma \ref{FractionalIdeals}\eqref{ProdExtend} that 
	\begin{equation*}
		I\intL = \prod_{n=1}^N (I_n\intL) \quad\mbox{and}\quad I'\intL = \prod_{n=1}^N (I'_n\intL).
	\end{equation*}
	
	We have assumed that $K$ is the Hilbert class field of $K$, so according to \cite{Furtwangler}, $I_n\intL$ and $I'_n\intL$ are principal ideals.  Hence, we may let $c_n$ and $c'_n$ be generators of
	$I_n\intL$ and $I'_n\intL$, respectively.  Again applying Lemma \ref{FractionalIdeals}, we conclude that
	\begin{equation*}
		\alpha\intL = A\intL = \frac{I\intL}{I'\intL} = \prod_{n=1}^N \frac{I_n\intL}{I'_n\intL} = \prod_{n=1}^N \frac{c_n\intL}{c'_n\intL} = \left(\prod_{n=1}^N \frac{c_n}{c'_n}\right)\intL,
	\end{equation*}
	so that $\alpha$ and $\prod_{n=1}^N c_n/c'_n$ are generators of the same fractional ideal of $\intL$.  Hence, there exists a unit $u\in \intL$ such that
	\begin{equation} \label{FirstProduct}
		\alpha = u \cdot  \prod_{n=1}^N \frac{c_n}{c'_n}.
	\end{equation}
	
	Additionally, we know that $J_n\intL$ and $J'_n\intL$ are principal and we shall let $d_n$ and $d'_n$ be their respective generators.  From this information, we deduce that 
	\begin{equation*}
		\left( \frac{d_n}{d'_n}\right) \intL = \frac{d_n\intL}{d'_n\intL} = \frac{J_n\intL}{J'_n\intL} = B_n\intL = (\beta_n\intK)\intL = \beta_n\intL
	\end{equation*}
	and there must exist a unit $y_n\in \intL$ such that 
	\begin{equation} \label{DBeta}
		y_nd_n/d'_n = \beta_n.  
	\end{equation}	
	Of course, $y_nd_n$ must also be a generator of $J_n\intL$.
	Using \eqref{IContainments}, we see that $J_n\intL \subseteq I_n\intL$ and $J'_n\intL \subseteq I'_n\intL$, and therefore, there must exist $r_n,r'_n\in \intL$ such that 
	\begin{equation} \label{CD}
		y_nd_n = c_nr_n\quad\mbox{and}\quad d'_n = c'_nr'_n.
	\end{equation}
	Since $L$ is assumed to be balanced, there exist units $u_n,u'_n\in \intL$ such that $|u_nr_n|_v,|u'_nr'_n|_v \geq 1$ for all $v\mid\infty$.
	Then applying \eqref{FirstProduct} we get that
	\begin{equation*}
		\alpha = u\cdot \frac{\prod_{n=1}^N u_n}{\prod_{n=1}^N u'_n} \cdot \prod_{n=1}^N \frac{c_n/u_n}{c'_n/u'_n}.
	\end{equation*}
	We now define 
	\begin{equation*}
		\gamma_0 = u\cdot \frac{\prod_{n=1}^N u_n}{\prod_{n=1}^N u'_n}\quad\mbox{and}\quad \gamma_n = \frac{c_n/u_n}{c'_n/u'_n}\quad\mbox{ for all } 1\leq n\leq N.
	\end{equation*}
	We obtain the conclusions \eqref{Product} and \eqref{Unit} immediately, so it remains to establish \eqref{Measures}.
	
	To see this, we have assumed that $J_n$ and $J'_n$ are relatively prime so that $J_n + J'_n = \intK$.  It now follows from Lemma \ref{FractionalIdeals} that
	\begin{equation} \label{RelativelyPrimes}
		J_n\intL + J'_n\intL = \intL\quad\mbox{and}\quad I_n\intL + I'_n\intL = \intL.
	\end{equation}
	Since $c_n$ and $c'_n$ are algebraic integers, we know that $\max\{|c_n|_v,|c'_n|_v\} \leq 1$ for all non-Archimedean places $v$ of $L$.
	If there exists a place $v\nmid\infty$ such that $\max\{|c_n|_v,|c'_n|_v\} < 1$ then $c_n$ and $c'_n$ would both belong to the maximal ideal
	$\mathcal M_v = \{x\in \intL: |x|_v < 1\}$.  In particular, we would have $I_n\intL + I'_n\intL  \subseteq \mathcal M_v$, contradicting the right hand equality of \eqref{RelativelyPrimes}.  
	As a result, we must have that $\max\{|c_n|_v,|c'_n|_v\} = 1$ for all non-Archimedean places $v$ of $L$ and we deduce that
	\begin{equation*}
		\max\left\{ \left|\frac{c_n}{u_n}\right|_v, \left|\frac{c'_n}{u'_n}\right|_v\right\}= 1\quad\mbox{ for all } v\nmid\infty.
	\end{equation*}
	By a similar argument, we obtain that 
	\begin{equation*}
		\max\{|y_nd_n|_v,|d'_n|_v\} = 1\quad\mbox{ for all } v\nmid\infty.
	\end{equation*}	
	Using these observations in conjunction with the product formula, for $1\leq n\leq N$ we have that
	\begin{equation*}
		\exp h(\gamma_n) = \prod_v \max\left\{ \left|\frac{c_n}{u_n}\right|_v, \left|\frac{c'_n}{u'_n}\right|_v\right\}
				=  \prod_{v\mid\infty}  \max\left\{ \left|\frac{c_n}{u_n}\right|_v, \left|\frac{c'_n}{u'_n}\right|_v\right\}
				\leq  \prod_{v\mid\infty} \max\left\{ \left|\frac{c_nu_nr_n}{u_n}\right|_v, \left|\frac{c'_nu'_nr_n}{u'_n}\right|_v\right\},
	\end{equation*}
	where the last inequality follows from the fact that $|u_nr_n|_v,|u'_nr'_n|_v \geq 1$ for all $v\mid\infty$.  From \eqref{CD} we get
	\begin{equation*}
		\exp h(\gamma_n) \leq \prod_{v\mid\infty} \max\{|c_nr_n|_v,|c'_nr'_n|_v\} = \prod_{v\mid\infty} \max\{|y_nd_n|_v,|d'_n|_v\} = \prod_v \max\{|y_nd_n|_v,|d'_n|_v\},
	\end{equation*}
	and the product formula along with \eqref{DBeta} yields
	\begin{equation*}
		\exp h(\gamma_n)\leq \prod_{v}\max\left\{ 1,\left| \frac{y_nd_n}{d'_n}\right|\right\} = \prod_{v} \max \left\{1,|\beta_n|_v\right\} = \exp h(\beta_n).
	\end{equation*}
	Finally, we see that
	\begin{equation*}
		h(\beta_n)  = h(\norm_{K(\alpha_n)/K}(\alpha_n)) \leq [K(\alpha_n):K]\cdot h(\alpha_n) = m_K(\alpha_n)
	\end{equation*}
	completing the proof.
\end{proof}

During the proof of Theorem \ref{GeneralReplacement}, we encountered a product of fractional ideals of the form $A = \prod_{n=1}^N I_n/I'_n$.  However, since the ideals on the right hand side
are not known to be principal, it is difficult to convert this information about ideals into information about elements.  Our remedy in the proof of Theorem \ref{GeneralReplacement} was to extend
each ideal to the Hilbert class field and use the fact that these extended ideals are principal.  An alternate approach is to raise both sides to a power equal to the class number of $K$.
Substituting this technique yields a variation on Theorem \ref{GeneralReplacement}.

\begin{thm} \label{PowerReplacement}
	Let $K$ be a balanced number field of class number $\lambda$.  Assume that $\alpha\in K^\times$ and $\alpha_1,\alpha_2,\ldots,\alpha_N\in\alg$ are such that 
	$\alpha = \alpha_1\alpha_2\cdots\alpha_N$.  Then there exist $\gamma_0,\gamma_1,\ldots,\gamma_N\in \alg$ satisfying the following four conditions:
	\begin{enumerate}[(i)]
		\item\label{Product2} $\alpha= \gamma_0\gamma_1\cdots\gamma_N$
		\item\label{Unit2} $\gamma_0$ is a unit in the ring of algebraic integers
		\item\label{Powers} $\gamma_n^\lambda \in K$ for all $0\leq n\leq N$
		\item\label{Measures2} $h(\gamma_n) \leq m_K(\alpha_n)$ for all $1\leq n\leq N$.
	\end{enumerate}
\end{thm}
\begin{proof}
	We begin the proof in the exact same way that we began the proof of Theorem \ref{GeneralReplacement} so we need not repeat all of the steps here.  We define $\beta_n\in \intK$ in the same
	way and fractional ideals $$A,B_1,B_2,\ldots,B_N,I,I',J_1,J_2,\ldots,J_N,J'_1,J'_2,\ldots,J'_N$$ also in the same way.  Just as before, we obtain ideals $I_1,I_2,\ldots,I_N,I'_1,I'_2,\ldots,I'_N$ of $\intK$ 
	satisfying the following properties:
	\begin{enumerate}[(a)]
		\item\label{IProducts2} $I=I_1I_2\cdots I_N$ and $I' = I'_1I'_2\cdots I'_N$
		\item\label{IContainments2} $J_n\subseteq I_n$ and $J'_n\subseteq I'_n$ for all $n$.
	\end{enumerate}
	Now instead of extending these ideals to $\intL$, we observe that
	\begin{equation} \label{LambdaProduct}
		I^\lambda = I_1^\lambda I_2^\lambda \cdots I_N^\lambda\quad\mbox{and}\quad (I')^\lambda = (I'_1)^\lambda (I'_2)^\lambda \cdots (I'_N)^\lambda.
	\end{equation}
	Since $\lambda$ is the class number of $\intK$, all of the ideals appearing in \eqref{LambdaProduct} are principal, so we may let $c_n$ and $c'_n$ be generators of 
	$I_n^\lambda$ and $(I'_n)^\lambda$, respectively.  As a result, we get that
	\begin{equation*}
		\alpha^\lambda \intK = A^\lambda  = \frac{I^\lambda}{(I')^\lambda} = \prod_{n=1}^N \frac{I_n^\lambda}{(I'_n)^\lambda} = \left(\prod_{n=1}^N \frac{c_n}{c'_n}\right)\intK
	\end{equation*}
	and there exists a unit $u\in \intK$ such that
	\begin{equation} \label{FirstProduct2}
		\alpha^\lambda = u\cdot \prod_{n=1}^N\frac{c_n}{c'_n}.
	\end{equation}
	Additionally, we let $d_n$ and $d'_n$ be generators of $J^\lambda_n$ and $(J'_n)^\lambda$, respectively, and we deduce that $\left(d_n/d'_n\right) \intK = \beta_n^\lambda \intK$.
	Therefore, there is a unit $y_n\in \intK$ such that $y_nd_n/d'_n = \beta^\lambda$, and of course, $y_nd_n$ is a generator of $J_n^\lambda$.   By \eqref{IContainments2}, 
	we know that $J_n^\lambda\subseteq I_n^\lambda$ and $(J'_n)^\lambda\subseteq (I'_n)^\lambda$.  Consequently, there exist $r_n,r'_n\in \intK$ such that
	$y_nd_n = c_nr_n$ and $d'_n = c'_nr'_n$.  Since $K$ is balanced, there exist units $u_n,u'_n\in \intK$ such that $|u_nr_n|_v,|u'_nr'_n|_v\geq 1$ for all $v\mid \infty$.
	Now applying \eqref{FirstProduct2}, we get that
	\begin{equation*}
		\alpha^\lambda = u\cdot \frac{\prod_{n=1}^N u_n}{\prod_{n=1}^N u'_n} \cdot \prod_{n=1}^N \frac{c_n/u_n}{c'_n/u'_n}.
	\end{equation*}
	Now select $\delta_0,\delta_1,\ldots,\delta_N\in \alg$ such that
	\begin{equation*}
		\delta^\lambda_0 = u\cdot \frac{\prod_{n=1}^N u_n}{\prod_{n=1}^N u'_n}\quad\mbox{and}\quad \delta^\lambda_n =\frac{c_n/u_n}{c'_n/u'_n}\quad\mbox{ for all } 1\leq n\leq N,
	\end{equation*}
	so we get that $\alpha^\lambda = \delta_0^\lambda\delta_1^\lambda \cdots \delta_N^\lambda$.  As a result, there exists a $\lambda$th root of unity $\zeta$ such that 
	$\alpha = \zeta \delta_0\delta_1\cdots \delta_N$.  We set
	\begin{equation*}
		\gamma_0 = \zeta\delta_0\quad\mbox{and}\quad \gamma_n = \delta_n \quad\mbox{ for all } 1\leq n\leq N,
	\end{equation*}
	and we immediately obtain properties \eqref{Product2}, \eqref{Unit2} and \eqref{Powers}.  To establish \eqref{Measures2} we observe that $d_n$ and $d'_n$ generate a relatively prime
	pair of ideals of $\intK$ just as they did in the proof of Theorem \ref{GeneralReplacement}.  Of course, $c_n$ and $c'_n$ also generate a relatively prime pair of ideals, so we find that
	\begin{align*}
		h(\gamma_n)  & = \frac{1}{\lambda} h(\gamma_n^\lambda) \\
			& = \frac{1}{\lambda} \sum_{v} \log\max\left\{ \left|\frac{c_n}{u_n}\right|_v, \left|\frac{c'_n}{u'_n}\right|_v\right\} \\
			& = \frac{1}{\lambda} \sum_{v\mid\infty} \log\max\left\{ \left|\frac{c_n}{u_n}\right|_v, \left|\frac{c'_n}{u'_n}\right|_v\right\} \ \\
			& \leq \frac{1}{\lambda}\log \max\{|c_nr_n|_v,|c'_nr'_n|_v\} \\
			&  = \frac{1}{\lambda}\log \max\{|y_nd_n|_v,|d'_n|_v\} \\
			& = \frac{1}{\lambda} h\left(\frac{y_nd_n}{d'_n}\right) \\
			& = h(\beta_n)
	\end{align*}
	We still have that $h(\beta_n) \leq m_K(\alpha_n)$ so the result follows.
\end{proof}

The advantage of Theorem \ref{PowerReplacement} over Theorem \ref{GeneralReplacement} is that its hypotheses only require that $K$ be balanced.  On the other hand, we have little control over the
elements $\gamma_n$.  Indeed, they could generate an extension of $K$ of degree larger than $\lambda$.  In any case, Theorems \ref{GeneralReplacement} and \ref{PowerReplacement} are equivalent 
when $\lambda = 1$.  In particular, Corollary \ref{ClassNumber1} is also a consequence of Theorem \ref{PowerReplacement}
and we could have constructed the examples of Section \ref{Examples} equally well using Theorem \ref{PowerReplacement} instead of Theorem \ref{GeneralReplacement}.

\section{Proofs Related to our Examples}

We conclude this article by giving the proofs of the results needed to provide the examples in Section \ref{Examples}. 

\begin{proof}[Proof of Lemma \ref{Rank1}]
	Suppose that $w_1$ and $w_2$ are the Archimedean places of $K$.  Since $\xi$ is not a root of unity we know that $h(\xi) > 0$ and we may assume without loss of generality
	that $|\xi|_{w_1} > 1$.  From the product formula, we then get $|\xi|_{w_2} < 1$ and $h(\xi) = \log |\xi|_{w_1}$.
	
	Let $\ell$ be the smallest integer such that $|\xi^\ell x|_{w_1} \geq 1$ so that $|\xi^{\ell - 1} x|_{w_1} < 1$.  As a result, we have
	\begin{equation*}
		1 < \left| \frac{1}{\xi^{\ell - 1}}\cdot \frac{1}{x} \right|_{w_1} = \left| \frac{\xi}{\xi^{\ell }}\cdot \frac{1}{x} \right|_{w_1}\cdot \left| \frac{x}{x}\right|_{w_2}
			 = \frac{|x|_{w_2}}{|\xi^\ell |_{w_1}}\cdot \frac{|\xi|_{w_1}}{|x|_{w_1}|x|_{w_2}}
	\end{equation*}
	But we also know that 
	\begin{equation*}
		|x|_{w_1}|x|_{w_2} = \prod_{v\mid\infty} |x|_v = |\norm_{K/\rat}(x)|_\infty^{1/[K:\rat]},
	\end{equation*}
	and we get that
	\begin{equation*}
		1 < \frac{|x|_{w_2}}{|\xi^\ell |_{w_1}}\cdot \frac{\exp h(\xi)}{|\norm_{K/\rat}(x)|_\infty^{1/[K:\rat]}} \leq \frac{|x|_{w_2}}{|\xi^\ell |_{w_1}} =
			\frac{|\xi^\ell x|_{w_2}}{|\xi^\ell |_{w_1}\cdot |\xi^\ell |_{w_2}} = |\xi^\ell x|_{w_2},
	\end{equation*}
	where the last equality follows from the fact that $\xi^{\ell}$ is a unit.  As a result, we have that $|\xi^\ell x|_v \geq 1$ for all $v\mid\infty$.

	For the second statement of the lemma, assume that $x\in \intK$.  If $x$ is a unit then we may use $u = x^{-1}$ to satisfy the definition of balanced.  Otherwise,
	we have
	\begin{equation*}
		\log |\norm_{K/\rat}(x)|_\infty \geq \log \min\{|\norm_{K/\rat}(y)|_\infty: y\in \intK\setminus\intK^\times\} \geq [K:\rat]\cdot h(\xi)
	\end{equation*}
	and the result follows from the first statement of the lemma.
\end{proof}

\begin{proof}[Proof of Corollary \ref{KInfimumWeak}]
	Since $\rank(K) = 1$, we may let $\varepsilon$ be a fundamental unit of $K$.  Among all units in $K$ that are not roots of unity, $\varepsilon$ certainly has the smallest
	Weil height, and therefore, we conclude that
	\begin{equation} \label{MiniHeight}
		1 < [K:\rat]\cdot h(\varepsilon) \leq \log\min\{|\norm_{K/\rat}(y)|_\infty: y\in \intK\setminus\intK^\times\}.
	\end{equation}
	Now assume that $\alpha_1,\alpha_2,\ldots,\alpha_N\in \alg$ such that $\alpha= \alpha_1\alpha_2\cdots\alpha_N$.
	By combining Lemma \ref{Rank1} and Theorem \ref{GeneralReplacement}, there must exist $\gamma_0,\gamma_1,\ldots,\gamma_N\in K$ such that
	\begin{enumerate}[(i)]
		\item $\alpha= \gamma_0\gamma_1\cdots\gamma_N$
		\item $\gamma_0$ is a unit in $\intL$
		\item $m_K(\gamma_n) \leq m_K(\alpha_n)$ for all $1\leq n\leq N$.
	\end{enumerate}
	Since $\alpha$ is not a unit, we shall assume without loss of generality that $\gamma_1$ is not a unit.  Since $\intK$ is a unique factorization domain and $K$ is its field of fractions, 
	we may choose $a,b\in \intK$ with $\gcd(a,b) = 1$ such that $\gamma_1 = a/b$.  In view of these assumptions, we get that $\max\{|a|_v,|b|_v\} = 1$ for all $v\nmid\infty$.
	Now using \eqref{MiniHeight} we obtain
	\begin{align*}
		h(\varepsilon) & \leq \log\max\{|\norm_{K/\rat}(a)|^{1/[K:\rat]}_\infty,|\norm_{K/\rat}(b)|^{1/[K:\rat]}_\infty\} \\
			& = \log \max\left\{ \prod_{v\mid\infty} |a|_v,\prod_{v\mid\infty} |a|_v\right\} \\
			& \leq \log\prod_{v\mid\infty}\{|a|_v,|b|_v\} \\
			& = \log\prod_{v}\{|a|_v,|b|_v\}.
	\end{align*}
	Now we apply the product formula to obtain $h(\varepsilon) \leq h(\gamma_1)$.  Since $\varepsilon,\gamma_1\in K$, this means that
	\begin{equation} \label{EpsilonBounds}
		m_K(\varepsilon) \leq m_K(\gamma_1)\quad\mbox{and}\quad m_K(\varepsilon^{-1}) \leq m_K(\gamma_1)
	\end{equation}
	and we have established the following inequalities:
	\begin{enumerate}[(i)]
		\item $\max\{m_K(\zeta),m_K(\gamma_1),\ldots,m_K(\gamma_N)\} \leq \max\{ m_K(\alpha_1),\ldots,m_K(\alpha_N)\}$
		\item $\max\{m_K(\zeta),m_K(\varepsilon),m_K(\gamma_1),\ldots,m_K(\gamma_N)\} \leq \max\{ m_K(\alpha_1),\ldots,m_K(\alpha_N)\}$
		\item $\max\{m_K(\zeta),m_K(\varepsilon^{-1}),m_K(\gamma_1),\ldots,m_K(\gamma_N)\} \leq \max\{ m_K(\alpha_1),\ldots,m_K(\alpha_N)\}$.
	\end{enumerate}
	Since $\varepsilon$ is a fundamental unit, we let $\ell\in \intg$ and $\zeta\in K^\times_\tors$ be such that $\gamma_0 = \zeta\varepsilon^\ell$, which yields
	\begin{equation*}
		\alpha = \zeta\varepsilon^\ell\gamma_1\gamma_2\cdots\gamma_N.
	\end{equation*}
	We have now found that
	\begin{equation*}
		m_{K,\infty}(\alpha) =  \inf\left\{\max_{1\leq n\leq N} \{m_K(\alpha_n)\}:N\in \nat,\ \alpha_n\in K,\ \alpha = \prod_{n=1}^N\alpha_n\right\}
	\end{equation*}
	and the result follows from Northcott's Theorem \cite{Northcott}.
\end{proof}

\end{document}